\documentclass{amsart}

\usepackage[T1]{fontenc}
\usepackage[utf8]{inputenc}
\usepackage{lmodern}
\usepackage{microtype}
\usepackage{amsmath,amssymb,amsthm,mathtools}
\usepackage{mathrsfs}
\usepackage{tikz-cd}
\usepackage[hidelinks]{hyperref}
\usepackage{enumitem}

\newtheorem{theorem}{Theorem}[section]
\newtheorem{proposition}[theorem]{Proposition}
\newtheorem{lemma}[theorem]{Lemma}
\newtheorem{corollary}[theorem]{Corollary}
\theoremstyle{definition}
\newtheorem{definition}[theorem]{Definition}
\theoremstyle{remark}
\newtheorem{remark}[theorem]{Remark}

\newcommand{\Pp}{\mathbf P}
\newcommand{\Gr}{\operatorname{Gr}}
\newcommand{\Flag}{\operatorname{Flag}}
\newcommand{\LG}{\operatorname{LG}}
\newcommand{\GL}{\operatorname{GL}}
\newcommand{\PGL}{\operatorname{PGL}}
\newcommand{\Sp}{\operatorname{Sp}}
\newcommand{\GSp}{\operatorname{GSp}}
\newcommand{\PSp}{\operatorname{PSp}}
\newcommand{\PGSp}{\operatorname{PGSp}}
\newcommand{\Spin}{\operatorname{Spin}}
\newcommand{\Aut}{\operatorname{Aut}}
\newcommand{\an}{\mathrm{an}}
\newcommand{\OO}{\mathcal O}
\newcommand{\Dcal}{\mathcal D}
\newcommand{\Fcal}{\mathcal F}
\newcommand{\Mcal}{\mathcal M}
\newcommand{\Tcal}{\mathcal T}

\title[Non-Archimedean Penrose correspondence]
{The Penrose Correspondence over a Non-Archimedean Field}

\author{Alberto Verjovsky}

\begin{document}

\begin{abstract}
Let \(K\) be a non-Archimedean local field of characteristic different
from \(2\), and let \(V\) have dimension four over \(K\). The Drinfeld upper half-spaces of \(V\) and \(\Lambda^2V\) define
compatible open analytic subsets of \(\Pp(V)\), \(\Gr(2,V)\), and
\(\Flag(1,2;V)\). Both restricted incidence projections are smooth and
surjective. For every torsion-free discrete subgroup of \(\PGL(V)(K)\),
the three actions are free and properly discontinuous, and the
incidence correspondence descends to smooth analytic quotients.
These domains contain an explicit flag defined over an unramified
extension of degree seven. A symplectic form on \(V\) equips the projective domain
with a contact structure and places the Grassmannian domain in
\(\Sp_4/(\Sp_2\times\Sp_2)\). For torsion-free discrete symplectic groups, this contact structure
descends to the quotient. Free discrete symplectic groups of every finite
rank at least two give contact threefolds whose underlying Berkovich
spaces have the homotopy type of a bouquet of circles. Every quotient
of the projective domain by a torsion-free discrete group preserving
the contact distribution is noncompact. For every torsion-free discrete
projective linear group, the flag quotient and all fibers of its two
incidence projections are noncompact; both projections are nonproper.

\end{abstract}



\maketitle

\begin{center}
\emph{To the memory of James Eells (1926--2007)}
\end{center}

\section{Introduction}

The smooth Penrose fibration \(\mathbf{CP}^3\to S^4\) is defined using
a quaternionic structure on \(\mathbf C^4\); see
\cite{Penrose,WardWells}. Its algebraic incidence diagram is
\[
\Pp(V)\xleftarrow{\,p\,}\Flag(1,2;V)\xrightarrow{\,q\,}\Gr(2,V),
\qquad \dim V=4.
\]
Here a flag is a pair \(L\subset W\subset V\) with
\(\dim L=1\) and \(\dim W=2\); the projections forget \(W\) and
\(L\), respectively. The Grassmannian is the Klein quadric in
\(\Pp(\Lambda^2V)\), and the fibers of \(p\) and \(q\) are
\(\Pp^2\) and \(\Pp^1\). Over \(\mathbf C\), the quaternionic lines
form a real submanifold of this Grassmannian isomorphic to \(S^4\).
Restricting the incidence diagram to that real submanifold recovers the
smooth twistor fibration. Eells and Salamon describe this bundle and
the Gauss lifts of maps from surfaces in
\cite[\S\S2, 4, 9]{EellsSalamon}.

The author's joint work with Jos\'e Seade on discrete actions on complex
projective spaces and complex Schottky groups
\cite{SeadeVerjovskyActions,SeadeVerjovskyKleinian,SeadeVerjovsky}
provides a geometric motivation for studying quotients of this diagram.
In particular, \cite{SeadeVerjovskyKleinian} uses twistor fibrations to
lift higher-dimensional Kleinian actions. Over a non-Archimedean local
field, Drinfeld upper half-spaces provide domains on which torsion-free
discrete projective linear groups act freely and properly
discontinuously. For the incidence diagram, the domains must also be compatible with
both projections and preserve their surjectivity.

Let \(K\) be a non-Archimedean local field with
\(\operatorname{char}K\ne2\): a complete discretely valued field
with nontrivial valuation and finite residue field. For a finite-dimensional
\(K\)-vector space \(E\), set
\[
\Omega_{\mathrm{Dr}}(E)=\Pp(E)^{\an}\setminus
\bigcup_{\substack{H\subset E\\\dim H=\dim E-1}}\Pp(H)^{\an}.
\]
All hyperplanes in this union are defined over \(K\), and \(\an\)
denotes Berkovich analytification. This is the Drinfeld upper half-space,
an open analytic subset of dimension \(\dim E-1\). If \(L/K\) is a
complete extension and \([a_1:\cdots:a_d]\in\Pp(E)(L)\) is written
in a \(K\)-basis, it belongs to \(\Omega_{\mathrm{Dr}}(E)\) precisely
when \(a_1,\ldots,a_d\) are linearly independent over \(K\).

Put \(\Tcal=\Pp(V)\), \(\Mcal=\Gr(2,V)\), and
\(\Fcal=\Flag(1,2;V)\).
\begin{theorem}[Incidence domains and quotients]\label{thm:D}
Let \(\dim_K V=4\). Define
\[
\Omega_{\Tcal}=\Omega_{\mathrm{Dr}}(V),\qquad
\Omega_{\Mcal}=\Mcal^{\an}\cap\Omega_{\mathrm{Dr}}(\Lambda^2V),
\qquad
\Omega_{\Fcal}=(p^{\an})^{-1}(\Omega_{\Tcal})
\cap(q^{\an})^{-1}(\Omega_{\Mcal}).
\]
These are nonempty open analytic subsets of \(\Tcal^{\an}\),
\(\Mcal^{\an}\), and \(\Fcal^{\an}\), of dimensions \(3,4,5\).
The restricted maps
\[
p:\Omega_{\Fcal}\longrightarrow\Omega_{\Tcal},\qquad
q:\Omega_{\Fcal}\longrightarrow\Omega_{\Mcal}
\]
are smooth and surjective, of relative dimensions \(2\) and \(1\).
For every torsion-free discrete subgroup
\(\Gamma\subset\PGL(V)(K)\), the three actions are free and properly
discontinuous. Their smooth analytic quotients exist, and the induced maps
\[
\Gamma\backslash\Omega_{\Tcal}
\xleftarrow{\ \bar p\ }\Gamma\backslash\Omega_{\Fcal}
\xrightarrow{\ \bar q\ }\Gamma\backslash\Omega_{\Mcal}
\]
are smooth and surjective, with the same relative dimensions.
\end{theorem}

A nondegenerate alternating form \(\omega\) on \(V\) defines a
contact distribution on \(\Pp(V)\). Write
\[
\Mcal^\circ=\{W\in\Gr(2,V):\omega|_W\text{ is nondegenerate}\}.
\]
The projective contact distribution \(\mathcal H\subset T_{\Pp(V)}\)
has fiber
\(\mathcal H_{[L]}=\operatorname{Hom}(L,L^{\perp_\omega}/L)\),
where \(L^{\perp_\omega}=\{v\in V:\omega(v,L)=0\}\).
A contact distribution on a smooth threefold is a rank-two subbundle
locally given by \(\ker\theta\), where \(\theta\wedge d\theta\)
is nowhere zero.

\begin{corollary}[Symplectic quotients]\label{thm:C}
For every torsion-free discrete subgroup \(\Gamma\) of the projective
image of \(\Sp(V,\omega)(K)\), the quotient
\(Y_\Gamma=\Gamma\backslash\Omega_{\Tcal}\) carries the descended
contact distribution. Moreover,
\[
\Omega_{\Mcal}\subset(\Mcal^\circ)^{\an},\qquad
\Mcal^\circ\simeq\Sp_4/(\Sp_2\times\Sp_2)
\]
as \(K\)-varieties. Thus the right-hand quotient in
Theorem~\ref{thm:D} is a quotient of a domain in this homogeneous space.
For every \(r\ge2\), such a group can be chosen isomorphic to \(F_r\), the free group on
\(r\) generators. The underlying Berkovich space of \(Y_\Gamma\) then has
the homotopy type of a bouquet of \(r\) circles; in particular,
\[
\pi_1(Y_\Gamma)\simeq F_r,\qquad
H_1(Y_\Gamma,\mathbf Z)\simeq\mathbf Z^r.
\]
Here \(\pi_1\) and \(H_1\) are the ordinary topological fundamental
group and singular homology.
\end{corollary}

Let \(H_\omega\subset\PGL(V)(K)\) be the stabilizer of
\([\omega]\in\Pp(\Lambda^2V^*)(K)\), where
\(g\cdot[\beta]=[(g^{-1})^*\beta]\). It is the projective image of
\(\GSp(V,\omega)(K)\), and Proposition~\ref{prop:symmetry} identifies
it with the group preserving \(\mathcal H\). A subgroup \(\Gamma\)
of a locally compact group \(G\) is called cocompact when
\(\Gamma\backslash G\) is compact.

\begin{theorem}[Noncompactness]\label{thm:N}
\begin{enumerate}[label=\textup{(\alph*)}]
\item For every torsion-free discrete \(\Gamma\subset H_\omega\),
the contact quotient \(Y_\Gamma\) is noncompact.
\item Every cocompact subgroup of \(\PGL(V)(K)\) contains an element
that fails to preserve \(\mathcal H\). For a torsion-free discrete
cocompact subgroup, the restriction \(\mathcal H|_{\Omega_{\Tcal}}\)
therefore does not descend to \(Y_\Gamma\).
\item For every torsion-free discrete \(\Gamma\subset\PGL(V)(K)\),
each fiber of \(\bar p\) and \(\bar q\) is closed and noncompact.
Consequently \(\Gamma\backslash\Omega_{\Fcal}\) is noncompact,
and both \(\bar p\) and \(\bar q\) are nonproper.
\end{enumerate}
\end{theorem}

The algebraic orbit decomposition and the projective contact structure
are classical; see \cite{Buczynski,LandsbergManivel}. The quotient and
retraction results for Drinfeld upper half-spaces come from
\cite{Drinfeld,Mustafin,Berkovich1995,Berkovich1999}.
Voskuil defines flag domains using stability for every maximal
\(K\)-split torus \cite[\S1.6 and Introduction]{Voskuil}.
For the Pl\"ucker embedding of \(\Gr(2,V)\),
\(\Omega_{\Mcal}\) is a proper subset of that stable locus
(Proposition~\ref{prop:stable-comparison}). The two incidence projections
on the domains defined above are surjective by
Lemmas~\ref{lem:q-surj} and~\ref{lem:p-surj}.

The notation \(\PSp(V,\omega)\) denotes the image of
\(\Sp(V,\omega)(K)\) in \(\PGL(V)(K)\).

\section{The algebraic Penrose correspondence}

Let \(V\) be a four-dimensional vector space over a field \(k\) of characteristic different from \(2\).
Set
\[
\Tcal=\Pp(V),\qquad
\Mcal=\Gr(2,V),\qquad
\Fcal=\Flag(1,2;V).
\]
A point of \(\Fcal\) is a flag
\[
L\subset W\subset V,\qquad \dim L=1,\quad \dim W=2.
\]
The incidence projections are
\[
p(L,W)=L,\qquad q(L,W)=W.
\]

\begin{proposition}\label{prop:fibers}
The map \(q:\Fcal\to\Mcal\) is a Zariski locally trivial
\(\Pp^1\)-bundle and \(p:\Fcal\to\Tcal\) is a Zariski locally trivial
\(\Pp^2\)-bundle.  Hence
\[
\dim\Tcal=3,\qquad \dim\Mcal=4,\qquad \dim\Fcal=5.
\]
\end{proposition}

\begin{proof}
For \(W\in\Gr(2,V)\),
\[
q^{-1}(W)=\Pp(W)\simeq\Pp^1.
\]
For \(L\in\Pp(V)\), the two-planes containing \(L\) are the lines in
\(V/L\), so
\[
p^{-1}(L)=\Pp(V/L)\simeq\Pp^2.
\]
The local triviality follows from the tautological bundles.
\end{proof}

Choose a nonzero volume form
\[
\operatorname{vol}:\Lambda^4V\xrightarrow{\sim}k.
\]
The quadratic form
\[
Q(\alpha)=\frac12\operatorname{vol}(\alpha\wedge\alpha)
\]
on \(\Lambda^2V\) is nondegenerate.  The Pl\"ucker map identifies
\[
\Mcal=\Gr(2,V)
\]
with the smooth quadric
\[
Q^4=\{[\alpha]\in\Pp(\Lambda^2V):Q(\alpha)=0\}\subset\Pp^5.
\]
Indeed, for \(0\ne\alpha\in\Lambda^2V\),
\[
\alpha\wedge\alpha=0
\iff
\alpha\text{ is decomposable}.
\]

The algebraic dimensions of the two ends of the diagram are three and
four. Thus no surjective algebraic morphism \(\Pp^3_k\to Q^4_k\)
exists. The base \(S^4\) of the smooth Penrose fibration is the real
four-dimensional submanifold of quaternionic lines in
\(\Gr(2,\mathbf C^4)\).

\section{The symplectic orbits and their quadratic model}\label{sec:orbits}

Let \(k\) have characteristic different from \(2\), and fix a
nondegenerate alternating form \(\omega\) on \(V\). A two-plane is
\emph{Lagrangian} when \(\omega\) vanishes on it. Set
\[
\Dcal=\LG(2,V),\qquad
\Mcal^\circ=\Gr(2,V)\setminus\Dcal.
\]
\begin{proposition}\label{prop:orbits}
The varieties \(\Dcal\) and \(\Mcal^\circ\) are respectively the
closed and open symplectic homogeneous subvarieties of \(\Gr(2,V)\).
Over \(k\),
\[
\Dcal\simeq Q^3_{\mathrm{split}},\qquad
\Mcal^\circ\simeq\Sp_4/(\Sp_2\times\Sp_2).
\]
The symplectic group acts transitively on the rational points of each
of these varieties over every field extension of \(k\).
\end{proposition}
\begin{proof}
An alternating form on a two-plane has rank zero or two. A basis of a
Lagrangian plane extends to a symplectic basis, which proves transitivity
on Lagrangian planes. For \(W\in\Mcal^\circ\),
\(V=W\oplus W^{\perp_\omega}\) is a decomposition into symplectic
planes. Choosing a symplectic basis of each summand proves transitivity
on the open locus. Its stabilizer is
\(\Sp(W)\times\Sp(W^{\perp_\omega})\); these arguments hold after
arbitrary scalar extension and identify the homogeneous varieties over
\(k\). In a symplectic basis with
\(\omega(e_1,e_4)=\omega(e_2,e_3)=1\), let \(p_{ij}\) be the
coefficients of a bivector in the basis \(e_i\wedge e_j\).
The Lagrangian condition is \(p_{23}=-p_{14}\), so the Pl\"ucker
equation becomes \(p_{12}p_{34}-p_{13}p_{24}-p_{14}^2=0\), a split
quadratic form of rank five.
\end{proof}

Choose a basis with \(\omega(e_1,e_4)=\omega(e_2,e_3)=1\), all other
pairings for \(i<j\) zero, and normalize \(\operatorname{vol}(e_1\wedge e_2\wedge e_3\wedge e_4)=1\).
Write \(\alpha=\sum_{i<j}p_{ij}e_i\wedge e_j\). Then
\[
Q(\alpha)=p_{12}p_{34}-p_{13}p_{24}+p_{14}p_{23},\qquad
c_\omega(\alpha)=p_{14}+p_{23}.
\]
Put \(b=e_1\wedge e_4+e_2\wedge e_3\), \(E=\ker c_\omega\), and
\(q_E=-Q|_E\). On \(E\), write \(p_{14}=s\), \(p_{23}=-s\); thus
\[
q_E=s^2-p_{12}p_{34}+p_{13}p_{24}.
\]
This is a nondegenerate split five-dimensional quadratic form (Witt
index two). The equation \(q_E=0\) in \(\Pp(E)\) gives \(\Dcal\).

\begin{proposition}\label{prop:quadric-model}
There is an isomorphism over \(k\)
\[
\Mcal^\circ\longrightarrow\{\beta\in E:q_E(\beta)=1\},
\qquad [\alpha]\longmapsto\frac{2\alpha}{c_\omega(\alpha)}-b.
\]
Its inverse sends \(\beta\) to \([\beta+b]\).
\end{proposition}
\begin{proof}
The polar form \(B_Q(u,v)=Q(u+v)-Q(u)-Q(v)\) satisfies
\(B_Q(\alpha,b)=c_\omega(\alpha)\), and \(Q(b)=1\).
For a decomposable \(\alpha\) with \(c_\omega(\alpha)\ne0\), the
vector \(\beta=2\alpha/c_\omega(\alpha)-b\) belongs to \(E\) and
has \(Q(\beta)=-1\). Conversely, if \(\beta\in E\) has
\(q_E(\beta)=1\), then \(Q(\beta+b)=0\) and
\(c_\omega(\beta+b)=2\). Thus \(\beta+b\) is nonzero and decomposable.
The two formulas are regular and inverse to each other.
\end{proof}

Let \(e=e_1\wedge e_4-e_2\wedge e_3\in E\). Then
\(q_E(e)=1\), and \(e^\perp\) is the split four-dimensional quadratic
space with form \(-p_{12}p_{34}+p_{13}p_{24}\). The classical split
exceptional isomorphisms are
\[
\Spin(E,q_E)\simeq\Sp_4,\qquad
\Spin(e^\perp,q_E|_{e^\perp})\simeq\Sp_2\times\Sp_2.
\]
The first is realized by the action on primitive bivectors \(E\);
its map to \(\mathrm{SO}(E,q_E)\) has kernel \(\{1,-1\}\).
Under this identification, the stabilizer of \(e\) is the product
preserving the planes \(\langle e_1,e_4\rangle\) and
\(\langle e_2,e_3\rangle\). Hence
\[
\Mcal^\circ\simeq
\Spin(E,q_E)/\Spin(e^\perp,q_E|_{e^\perp}).
\]
Thus \(\Mcal^\circ\) is the homogeneous space
\(\Spin(5)/\Spin(4)\) for these split quadratic forms. Its quadric
model is affine.

\section{Legendrian and transverse lines}\label{sec:legendrian}

The symplectic form \(\omega\) defines the standard contact
distribution on \(\Tcal=\Pp(V)\); see
\cite{LandsbergManivel,Buczynski}.  At a line \(L\subset V\),
\[
T_{[L]}\Pp(V)\simeq\operatorname{Hom}(L,V/L).
\]
The contact plane is
\[
\mathcal H_{[L]}
=
\left\{
\varphi\in\operatorname{Hom}(L,V/L):
\omega(v,\widetilde{\varphi(v)})=0
\right\},
\]
where \(0\ne v\in L\) and
\(\widetilde{\varphi(v)}\) is any lift to \(V\).

On \(V\setminus\{0\}\), let \(\xi\) be the Euler vector field,
\(\xi_v=v\), and set
\[
\widetilde\theta=\iota_\xi\omega.
\]
Scalar multiplication by \(a\in k^\times\) pulls
\(\widetilde\theta\) back to \(a^2\widetilde\theta\), and
\(\iota_\xi\widetilde\theta=\omega(\xi,\xi)=0\).
These homogeneity and horizontality conditions give an
\(\OO_{\Pp(V)}(2)\)-valued one-form
\[
\theta:T_{\Pp(V)}\longrightarrow\OO_{\Pp(V)}(2)
\]
with kernel \(\mathcal H\).  In a local trivialization of \(\OO(2)\), let \(\theta_U\) denote
the corresponding ordinary one-form. The identity
\(d\widetilde\theta=2\omega\) and nondegeneracy of \(\omega\) give
\(\theta_U\wedge d\theta_U\ne0\). Changing the trivialization
multiplies this three-form by an invertible square. These local forms
therefore define a nowhere-zero section of
\(\Omega^3_{\Pp(V)}\otimes\OO(4)\), establishing the contact condition.
A smooth curve is called Legendrian when its tangent line lies in the
contact distribution at every point.

For \(W\in\Gr(2,V)\), let
\[
\ell_W=\Pp(W)\subset\Pp(V).
\]

\begin{proposition}\label{prop:line_dichotomy}
The line \(\ell_W\) is Legendrian if and only if \(W\) is Lagrangian.
If \(\omega|_W\) is nondegenerate, then \(\ell_W\) is everywhere
transverse to the contact distribution.
\end{proposition}

\begin{proof}
Fix \(L\subset W\), \(\dim L=1\).  The tangent line to \(\ell_W\) at
\([L]\) is
\[
T_{[L]}\ell_W\simeq\operatorname{Hom}(L,W/L).
\]
It is contained in \(\mathcal H_{[L]}\) precisely when
\[
\omega(L,W)=0.
\]
If \(W\) is Lagrangian this holds for every \(L\subset W\), so
\(\ell_W\) is Legendrian.

Assume that \(\omega|_W\) is nondegenerate.  In the symplectic
plane \(W\), the orthogonal complement of a line \(L\) is \(L\) itself.
Hence
\[
\omega(L,W)\ne0
\]
for every \(L\subset W\).  Therefore
\(T_{[L]}\ell_W\not\subset\mathcal H_{[L]}\) at every point.
\end{proof}

Proposition~\ref{prop:line_dichotomy} is the line case of the standard
equivalence between projective Legendrian subvarieties and Lagrangian
affine cones; see \cite{Buczynski,LandsbergManivel}.  

\section{Projective symmetries}

\begin{proposition}\label{prop:symmetry}
A projective linear transformation preserves \(\mathcal H\) if and
only if it is induced by a symplectic similitude, that is, by
\(g\in\GL(V)\) with \(g^*\omega=c\omega\) for some \(c\in k^\times\).
The group of these transformations is \(\PGSp(V,\omega)\).
It preserves both symplectic orbits and the incidence diagram.
\end{proposition}
\begin{proof}
Let \(g\in\GL(V)\).  If
\[
g^*\omega=c\omega,\qquad c\in k^\times,
\]
then the contact one-form is multiplied by \(c\), and its kernel is
unchanged.  Thus the projective image of \(\GSp(V,\omega)\) preserves
\(\mathcal H\).

Conversely, suppose that the projective transformation induced by
\(g\) preserves \(\mathcal H\), and set
\[
\beta:=g^*\omega.
\]
By Proposition~\ref{prop:line_dichotomy}, a projective line
\(\Pp(\langle v,w\rangle)\) is Legendrian precisely when
\(\omega(v,w)=0\).  Since \(g\) preserves both projective lines and
the contact distribution,
\[
\omega(v,w)=0
\quad\Longleftrightarrow\quad
\beta(v,w)=0
\qquad(v,w\in V).
\]
Fix \(0\ne v\in V\).  The two nonzero linear forms
\[
\alpha_v:=\omega(v,\cdot),\qquad
\beta_v:=\beta(v,\cdot)
\]
have the same kernel; hence
\[
\beta_v=c(v)\alpha_v
\]
for some \(c(v)\in k^\times\).  If \(u,v\) are linearly independent,
then \(\alpha_u\) and \(\alpha_v\) are linearly independent because
\(\omega\) is nondegenerate.  Using
\[
\beta_{u+v}=\beta_u+\beta_v
\]
and comparing the coefficients of \(\alpha_u\) and \(\alpha_v\)
gives
\[
c(u)=c(v)=c(u+v).
\]
If \(u\) and \(v\) are proportional, choose a vector independent of
both and apply the preceding argument twice.  Thus \(c(v)\) is
independent of \(v\ne0\); write it as \(c\).  Therefore
\[
g^*\omega=\beta=c\omega,
\]
so \(g\) is a symplectic similitude and
\[
\Aut_{\mathrm{proj}}(\Pp(V),\mathcal H)=\PGSp(V,\omega).
\]

A symplectic similitude preserves the condition
\(\omega|_W=0\), hence preserves \(\LG(2,V)\) and its complement.
Finally, every projective linear transformation sends flags
\(L\subset W\) to flags \(gL\subset gW\), so the incidence maps are
equivariant.
\end{proof}

Over an algebraically closed field, the projective images of
\(\GSp_4\) and \(\Sp_4\) coincide; see \cite{Buczynski}.

\section{Berkovich analytification}

Let \(k\) be a complete non-Archimedean field of characteristic
different from \(2\). Berkovich analytification
\cite[\S3.4]{Berkovich1990} gives
\[
\begin{tikzcd}[column sep=large]
& \Fcal^{\an}\arrow[dl,"p^{\an}"']\arrow[dr,"q^{\an}"] &\\
\Tcal^{\an} && \Mcal^{\an}.
\end{tikzcd}
\]

\begin{proposition}\label{prop:analytic_fibers}
For \(x\in\Mcal^{\an}\), with completed residue field
\(\mathscr H(x)\),
\[
(q^{\an})^{-1}(x)
\simeq
(\Pp^1_{\mathscr H(x)})^{\an}.
\]
For \(z\in\Tcal^{\an}\),
\[
(p^{\an})^{-1}(z)
\simeq
(\Pp^2_{\mathscr H(z)})^{\an}.
\]
\end{proposition}

\begin{proof}
Analytification is compatible with fiber products and extension of the
ground field; see \cite[\S3.4]{Berkovich1990}.  Base change of the
algebraic projective bundles to the completed residue field
\(\mathscr H(x)\) or \(\mathscr H(z)\) therefore gives the stated
analytic fibers.
\end{proof}

The algebraic contact distribution also analytifies.  Consequently the
classification of lines in Proposition~\ref{prop:line_dichotomy} is compatible
with extension of the ground field and with analytification of the
corresponding algebraic lines.

The decomposition into symplectic homogeneous subvarieties gives
\[
\Gr(2,V)^{\an}=(\Mcal^\circ)^{\an}\sqcup\Dcal^{\an}.
\]
The quadratic model of Section~\ref{sec:orbits} and its homogeneous-space
interpretation hold over \(k\) and therefore analytify over \(k\).
The projective quadric is compact in the Berkovich topology.

\section{Free discrete symplectic groups}

Let \(K\) be a non-Archimedean local field with
\(\operatorname{char}K\ne2\), of arbitrary residue characteristic.
Fix a uniformizer \(\pi\), a generator of the maximal ideal of its
valuation ring, and choose a symplectic basis
\[
e_1,e_2,e_3,e_4
\]
such that
\[
\omega(e_1,e_4)=1,\qquad
\omega(e_2,e_3)=1
\]
and \(\omega(e_i,e_j)=0\) for all other pairs with \(i<j\).

\begin{definition}
An element \(g\in\PGL_n(K)\) is \emph{proximal} if a representative
in \(\GL_n(K)\) has a unique
eigenvalue of maximal absolute value, this eigenvalue is simple and
belongs to \(K\), and its eigenline is therefore a \(K\)-rational
point of \(\Pp^{n-1}\).  Its \emph{attracting point} is that
eigenline and its \emph{repelling hyperplane} is the projectivization
of the sum of the remaining generalized eigenspaces.  It is
\emph{biproximal} if both \(g\) and \(g^{-1}\) are proximal.
\end{definition}

Consider
\[
a=
\operatorname{diag}(\pi^{-3},\pi^{-1},\pi,\pi^3)
\in\Sp_4(K).
\]
The projective transformation induced by \(a\) is biproximal: its
attracting point is \([e_1]\), its repelling hyperplane is
\(\Pp(\langle e_2,e_3,e_4\rangle)\), and the inverse has attracting
point \([e_4]\) and repelling hyperplane
\(\Pp(\langle e_1,e_2,e_3\rangle)\).

For the projective ping--pong criterion, see Tits~\cite{Tits}
and Breuillard--Gelander~\cite{BG}.

\begin{lemma}[Projective ping--pong]\label{lem:pingpong}
Let \(K\) be a local field and let \(g_1,\ldots,g_r\in\PGL_n(K)\) be
biproximal.  Suppose the \(2r\) attracting points of \(g_i^{\pm1}\) are pairwise
distinct and that the attracting point of \(g_i^\varepsilon\) does not
belong to the repelling hyperplane of \(g_j^\delta\) whenever the
letters \(g_i^\varepsilon\) and \(g_j^\delta\) are not inverse.
Then, after replacing all \(g_i\) by a common sufficiently large
positive power, they generate a free discrete subgroup of
\(\PGL_n(K)\).
\end{lemma}

\begin{proof}
Choose a point outside all the finitely many repelling hyperplanes
and attracting points; such a point exists because \(K\) is infinite.
Choose a sufficiently small open neighborhood \(U\) of it. For each
letter \(g_i^{\pm1}\), choose an attracting neighborhood
\(A_i^\pm\) and a neighborhood \(R_i^\pm\) of its repelling hyperplane
so that all attracting neighborhoods are pairwise disjoint and
\(A_i^\varepsilon\cap R_j^\delta=\varnothing\) whenever the letters
\(g_i^\varepsilon\) and \(g_j^\delta\) are not inverse.  Proximality
implies uniform contraction on compact subsets of the complement of the
repelling hyperplane.  After replacing the \(g_i\) by a common
sufficiently large positive power,
\[
g_i^\varepsilon\bigl(\Pp^{n-1}(K)\setminus R_i^\varepsilon\bigr)
\subset A_i^\varepsilon
\qquad(\varepsilon=\pm1).
\]
The ordinary ping--pong argument gives freeness.

All the neighborhoods can initially be chosen so that \(U\) is
disjoint from every \(A_i^\pm\) and \(R_i^\pm\).
The same inclusions show that every nontrivial reduced word \(w\)
satisfies
\[
w(U)\cap U=\varnothing.
\]
If distinct elements \(w_m\) converged to the identity in
\(\PGL_n(K)\), then \(w_m(U)\cap U\neq\varnothing\) for all sufficiently
large \(m\), a contradiction.  Hence the subgroup is discrete.
\end{proof}

\begin{proposition}\label{prop:free_sp}
For every \(r\ge2\) there are elements
\[
\gamma_1,\ldots,\gamma_r\in\PSp_4(K)
\]
which freely generate a discrete subgroup.
\end{proposition}

\begin{proof}
Take \(r\) symplectic conjugates
\[
a_i=h_i a h_i^{-1},\qquad h_i\in\Sp_4(K),
\]
with attracting points and repelling hyperplanes in the transverse
position required by Lemma~\ref{lem:pingpong}. Each equality between attracting
points or incidence
\[
x_i^\varepsilon\in H_j^\delta
\]
excluded by the lemma defines a proper Zariski closed subset of \(\Sp_4^r\).  For the incidence
conditions, after the other \(h_\nu\) are fixed, transitivity of
\(\Sp_4\) on the relevant symplectic flag data permits \(h_i\) to
move the eigenline away from the prescribed hyperplane; equality of attracting points is excluded in the same way.  The complement of the finite
union of these proper closed subsets is therefore a nonempty Zariski
open subset of \(\Sp_4^r\).  Since \(K\) is infinite and \(\Sp_4\) is a connected
\(K\)-rational algebraic group, \(\Sp_4(K)^r\) is Zariski dense and this
open set contains a \(K\)-rational tuple.  A sufficiently large common
power of the \(a_i\) then satisfies the lemma.  Since all generators lie in \(\PSp_4(K)\), so does the group
they generate.
\end{proof}

\section{Drinfeld upper half-space and contact quotients}

Let \(E\) be a finite-dimensional \(K\)-vector space, \(d=\dim E\ge2\),
and let \(\widehat\Omega_{\mathrm{Dr}}(E)\) denote the
Drinfeld--Mustafin formal model of \(\Omega_{\mathrm{Dr}}(E)\) over
\(\mathcal O_K\); see \cite{Drinfeld,Mustafin}.  Berkovich proves that
this formal scheme is strictly poly-stable and that its polysimplicial
set realizes the Bruhat--Tits building of \(\PGL(E)\); see
\cite[Remark~5.12(iii)]{Berkovich1999} and \cite{RTW}.

\begin{proposition}\label{prop:drinfeld_action}
Let \(E\) be as above and let
\(\Gamma\subset\PGL(E)(K)\) be torsion-free and discrete.  Then:
\begin{enumerate}[label=\textup{(\alph*)}]
\item \(\Gamma\) acts freely and properly discontinuously on
\(\Omega_{\mathrm{Dr}}(E)\);
\item the quotient
\(\Gamma\backslash\Omega_{\mathrm{Dr}}(E)\) exists as a smooth
\((d-1)\)-dimensional \(K\)-analytic space;
\item its underlying Berkovich space admits a strong deformation
retraction onto a closed subspace canonically homeomorphic to
\[
\Gamma\backslash\mathscr B(\PGL(E),K).
\]
\end{enumerate}
\end{proposition}

\begin{proof}
The locally finite Bruhat--Tits building is a proper metric space,
and the action of \(\PGL(E)(K)\) is proper, with compact point
stabilizers.  Since
\(\Gamma\) is discrete, its intersection with any such compact
subgroup is finite, and it is therefore trivial because \(\Gamma\)
is torsion-free.  Hence the action on the building is free and proper.

The Drinfeld--Mustafin formal model is \(\PGL(E)(K)\)-equivariant.
In particular, if an element of \(\Gamma\) fixes a point of
\(\Omega_{\mathrm{Dr}}(E)\), equivariance of the Berkovich retraction
forces it to fix the corresponding point of the building; hence it is
the identity.  Thus the analytic action is free.  For a compact subset \(C\) of the analytic domain, its image
under the continuous equivariant retraction is compact in the building.
If \(\gamma C\cap C\ne\varnothing\), the two corresponding compact
images also intersect. Properness on the building allows only finitely
many such \(\gamma\), proving properness of the analytic action.
Freeness and local compactness then give, at each point, an open
neighborhood disjoint from all its nonidentity translates.

The disjoint neighborhoods just obtained give analytic charts for
\(\Gamma\backslash\Omega_{\mathrm{Dr}}(E)\); transition maps are
restrictions of elements of \(\Gamma\). Properness makes the quotient
Hausdorff. The quotient map is locally an analytic isomorphism, so the
quotient is smooth of dimension \(d-1\). This also follows from the
analytic quotient construction in \cite{Berkovich1995}; see the discussion
after \cite[Corollary~8.6]{Berkovich1999}.

Finally, \cite[Remark~5.12(iii)]{Berkovich1999} identifies the
polysimplicial set of the formal model with the Bruhat--Tits building,
and \cite[Corollary~8.6]{Berkovich1999} gives the strong deformation
retraction of the quotient onto
\[
\Gamma\backslash\mathscr B(\PGL(E),K).
\]
\end{proof}

\begin{proof}[Proof of the contact and topological assertions of Corollary~\ref{thm:C}]
Let \(\Gamma\) be a torsion-free discrete symplectic group.
Proposition~\ref{prop:drinfeld_action}, applied to \(E=V\), gives the
free proper action, the smooth quotient, and the strong deformation
retraction onto \(\Gamma\backslash\mathscr B(\PGL(V),K)\).

The contact distribution \(\mathcal H\subset T_{\Pp^3_K}\) is preserved
by \(\PSp(K^4,\omega)\), hence its restriction to
\(\Omega_{\Tcal}\) is \(\Gamma\)-invariant.  Since the
quotient map is locally an analytic isomorphism, \(\mathcal H\)
descends uniquely to a rank-two analytic subbundle of \(T_{Y_\Gamma}\).
The contact condition is local and is preserved under pullback by local
analytic isomorphisms, so the descended distribution is again contact.

For every \(r\ge2\), Proposition~\ref{prop:free_sp} supplies a group
\(\Gamma\simeq F_r\). The Bruhat--Tits building is a contractible locally
finite polysimplicial complex. The action of this \(\Gamma\) is free and
proper by the proof of Proposition~\ref{prop:drinfeld_action}.
Therefore
\[
\Gamma\backslash\mathscr B
\]
is a classifying space \(B\Gamma\), hence an Eilenberg--MacLane space
of type \(K(\Gamma,1)\).  The strong deformation retraction gives
\[
\pi_1(Y_\Gamma)\simeq\Gamma\simeq F_r
\]
and
\[
H_1(Y_\Gamma,\mathbf Z)
\simeq
\Gamma^{\mathrm{ab}}
\simeq
\mathbf Z^r.
\]
Since a bouquet of \(r\) circles is also a classifying space for
\(F_r\), the two classifying CW complexes are homotopy equivalent.
Thus \(Y_\Gamma\) is homotopy equivalent to a bouquet of \(r\) circles.
\end{proof}

\begin{remark}
Every projective line defined over \(K\) is contained in a
\(K\)-rational hyperplane. Its analytification is therefore disjoint
from \(\Omega_{\Tcal}\).
\end{remark}

\section{Descent of the incidence correspondence}\label{sec:descent}

\begin{lemma}\label{lem:generic-analytic-point}
Let \(L\) be a complete non-Archimedean extension of \(K\), and let
\(X\) be an integral projective \(L\)-variety.  There exists a
Berkovich analytic point of \(X^{\an}\) whose image in the underlying
scheme \(X\) is the generic point.  Consequently such a point belongs
to the complement of every proper closed algebraic subset of \(X\).
\end{lemma}

\begin{proof}
Choose a transcendence basis \(t_1,\ldots,t_m\) of \(L(X)/L\).
Equip \(L(t_1,\ldots,t_m)\) with the multiplicative Gauss absolute
value, defined on polynomials by the maximum absolute value of their
coefficients. Extend it to the finite algebraic extension \(L(X)\)
and take the completion \(F\). This gives an isometric embedding
\(L\hookrightarrow F\) and an injective field map \(L(X)\hookrightarrow F\).
The resulting morphism \(\operatorname{Spec}F\to X\) defines a
Berkovich point over the generic point. Its kernel on every affine
coordinate ring meeting the generic point is zero, so it lies in no
proper algebraic closed subset.
\end{proof}

\begin{lemma}\label{lem:OmegaM-openorbit}
The domain \(\Omega_{\Mcal}\) is nonempty. For every nondegenerate
alternating form \(\omega\) on \(V\), it satisfies
\[
\Omega_{\Mcal}\subset(\Mcal^\circ)^{\an}.
\]
\end{lemma}

\begin{proof}
The symplectic form defines the nonzero \(K\)-linear functional
\[
c_\omega:\Lambda^2V\longrightarrow K,
\qquad
v\wedge w\longmapsto\omega(v,w).
\]
Its projectivized kernel is a \(K\)-rational hyperplane in
\(\Pp(\Lambda^2V)\), and
\[
\LG(2,V)
=
\Gr(2,V)\cap\Pp(\ker c_\omega).
\]
Since \(\Omega_{\mathrm{Dr}}(\Lambda^2V)\) avoids every
\(K\)-rational hyperplane,
\[
\Omega_{\Mcal}\cap\LG(2,V)^{\an}=\varnothing.
\]
Hence
\[
\Omega_{\Mcal}\subset(\Mcal^\circ)^{\an}.
\]

For nonemptiness, the Klein quadric \(\Gr(2,V)\subset
\Pp(\Lambda^2V)\) is integral and is not contained in any hyperplane.
Apply Lemma~\ref{lem:generic-analytic-point} to its generic point.  No
nonzero \(K\)-linear form on \(\Lambda^2V\) vanishes there, so the
resulting analytic point lies in
\(\Omega_{\mathrm{Dr}}(\Lambda^2V)\).
\end{proof}

\begin{lemma}\label{lem:q-surj}
The map
\[
q:\Omega_{\Fcal}\longrightarrow\Omega_{\Mcal}
\]
is surjective.
\end{lemma}

\begin{proof}
Let \(x\in\Omega_{\Mcal}\), put \(E=\mathscr H(x)\), and let
\[
W_x\subset V_E:=V\otimes_KE
\]
be the two-dimensional \(E\)-subspace represented by the
\(E\)-rational point induced by \(x\).

For every \(K\)-rational hyperplane \(H\subset V\),
\[
W_x\not\subset H_E.
\]
Indeed, if \(W_x\subset H_E\), then its Pl\"ucker line
\(\Lambda^2W_x\) lies in \((\Lambda^2H)_E\).  The projective plane
\(\Pp(\Lambda^2H)\subset\Pp(\Lambda^2V)\) is contained in a
\(K\)-rational hyperplane of \(\Pp(\Lambda^2V)\), contradicting
\(x\in\Omega_{\mathrm{Dr}}(\Lambda^2V)\).

Thus, for every \(K\)-rational hyperplane \(H\subset V\),
\[
\Pp(W_x)\cap\Pp(H_E)
\]
is a proper closed subset of the projective line \(\Pp(W_x)\).
By Lemma~\ref{lem:generic-analytic-point}, choose an analytic point
\(z\in\Pp(W_x)^{\an}\) lying over its generic point.  Then \(z\) lies
in none of these proper intersections.  Its image in
\(\Pp(V)^{\an}\) therefore belongs to
\(\Omega_{\Tcal}\).  The flag determined by \(z\subset W_x\) lies in
\(\Omega_{\Fcal}\) and maps to \(x\).
\end{proof}

\begin{lemma}\label{lem:p-surj}
The map
\[
p:\Omega_{\Fcal}\longrightarrow\Omega_{\Tcal}
\]
is surjective.
\end{lemma}

\begin{proof}
Let \(z\in\Omega_{\Tcal}\), put \(E=\mathscr H(z)\), and let
\(L_z\subset V_E\) be the corresponding one-dimensional
\(E\)-subspace.  The fiber of \(p\) over \(z\) is
\[
\Pp(V_E/L_z)\simeq\Pp^2_E,
\]
whose points correspond to two-planes \(W\subset V_E\) containing
\(L_z\).

Let \(A\subset\Lambda^2V\) be a \(K\)-rational hyperplane, defined by
a nonzero element
\[
\lambda\in(\Lambda^2V)^*.
\]
Regard \(\lambda\) as an alternating bilinear form on \(V\).  The
condition that the Pl\"ucker point of \(W=\langle L_z,v\rangle\) lie
in \(\Pp(A)\) is
\[
\lambda(\ell,v)=0,
\]
where \(0\ne\ell\in L_z\).  This is a linear condition on
\(V_E/L_z\).

This linear functional is nonzero. If it vanished identically, then
\[
\lambda(\ell,v)=0
\qquad\text{for every }v\in V_E,
\]
so \(L_z\) would be contained in the scalar extension of the radical
\[
\operatorname{rad}(\lambda)
=
\{u\in V:\lambda(u,V)=0\}.
\]
Because \(\lambda\ne0\) is alternating on a four-dimensional space and
\(\operatorname{char}K\ne2\), its radical has dimension \(0\) or \(2\).
If it has dimension \(2\), it is contained in a \(K\)-rational
hyperplane of \(V\); this would contradict
\(z\in\Omega_{\mathrm{Dr}}(V)\). If the radical has dimension zero,
the inclusion would force \(L_z=0\), again a contradiction.

Hence every \(K\)-rational hyperplane of
\(\Pp(\Lambda^2V)\) cuts the fiber \(\Pp(V_E/L_z)\) in a proper
closed algebraic subset.  Choose, by
Lemma~\ref{lem:generic-analytic-point}, an analytic point of this
fiber lying over its generic point.  Its associated two-plane has
Pl\"ucker point in
\(\Omega_{\mathrm{Dr}}(\Lambda^2V)\), hence in
\(\Omega_{\Mcal}\).  This gives a point of
\(\Omega_{\Fcal}\) over \(z\).
\end{proof}

\begin{lemma}\label{lem:exterior-square}
The exterior-square representation induces a closed immersion
\[
\iota:\PGL(V)\longrightarrow\PGL(\Lambda^2V).
\]
In particular, if \(\Gamma\subset\PGL(V)(K)\) is discrete, then
\(\iota(\Gamma)\) is discrete.
\end{lemma}

\begin{proof}
Suppose that \(g\in\GL(V)\) acts on \(\Lambda^2V\) as the scalar
\(c\).  For every nonzero \(v\in V\) and every \(w\notin kv\),
\[
gv\wedge gw=c\,v\wedge w.
\]
Hence \(g\) preserves every two-plane containing \(v\).  The
intersection of all such two-planes is the line \(kv\), so \(g\)
preserves every line in \(V\).  A linear transformation preserving
every line is scalar.  Thus the induced homomorphism
\[
\PGL(V)\longrightarrow\PGL(\Lambda^2V)
\]
has trivial kernel.

Its differential is also injective when \(\operatorname{char}K\ne2\):
if \(X\in\mathfrak{gl}(V)\) acts trivially on \(\Lambda^2V\)
modulo scalars, then after subtracting a scalar matrix one has
\[
Xv\wedge w+v\wedge Xw=0
\qquad(v,w\in V),
\]
which forces \(X=0\): in a basis, comparison of the coefficients
of \(e_k\wedge e_j\), with \(i,j,k\) distinct, first makes every
off-diagonal entry of \(X\) zero. The remaining equations are
\(x_i+x_j=0\) for all \(i\ne j\); since \(2\ne0\), all \(x_i\)
vanish. The same kernel argument works over an algebraic closure.  The image of a
homomorphism of algebraic groups is closed; trivial scheme-theoretic
kernel together with injective differential therefore makes this
homomorphism a closed immersion.  Consequently the induced map on
\(K\)-points is a topological embedding, and discreteness is preserved.
\end{proof}

\begin{proof}[Proof of Theorem~\ref{thm:D}]

Drinfeld upper half-spaces are open analytic subsets. Their inverse
images under the Pl\"ucker embedding and the incidence projections are
therefore open in the indicated smooth varieties. Lemma~\ref{lem:OmegaM-openorbit}
gives nonemptiness of \(\Omega_{\Mcal}\), and
Lemmas~\ref{lem:q-surj} and~\ref{lem:p-surj} give the surjectivity of
both maps. Nonemptiness of \(\Omega_{\Tcal}\) also follows from
Lemma~\ref{lem:generic-analytic-point} applied to projective space.
All three domains have the dimensions of their ambient varieties.
Since the algebraic projections are smooth projective bundles,
restriction to these open analytic subsets gives smooth maps of relative
dimensions two and one.

Every element of \(\PGL(V)(K)\) permutes the \(K\)-rational
hyperplanes of \(V\), and its exterior square permutes those of
\(\Lambda^2V\). Thus the domains are invariant. Proposition~\ref{prop:drinfeld_action}
gives a free proper action on \(\Omega_{\Tcal}\).
Lemma~\ref{lem:exterior-square} makes the image of \(\Gamma\) in
\(\PGL(\Lambda^2V)(K)\) torsion-free and discrete. The same
proposition gives a free proper action on
\(\Omega_{\mathrm{Dr}}(\Lambda^2V)\). Restriction to the invariant
analytic subspace \(\Omega_{\Mcal}\) preserves these properties.

On \(\Omega_{\Fcal}\), freeness follows from equivariance of \(p\).
If \(C\) is compact and \(\gamma C\cap C\ne\varnothing\), then
\(\gamma p(C)\cap p(C)\ne\varnothing\). Properness on
\(\Omega_{\Tcal}\) permits only finitely many such \(\gamma\).
Thus all three actions are free and proper. They admit open
neighborhoods disjoint from their nonidentity translates, and gluing
these neighborhoods gives the analytic quotients, with quotient maps
locally analytic isomorphisms.

Equivariance descends \(p\) and \(q\) to \(\bar p\) and \(\bar q\).
Their surjectivity follows from that of \(p\) and \(q\). Near any
flag, choose neighborhoods on which the source and target quotient maps
are isomorphisms, and shrink the source so its image lies in the chosen
target neighborhood. The descended map is then identified locally with
the restricted incidence map. It is smooth of the same relative
dimension. Finally, Lemma~\ref{lem:OmegaM-openorbit} and
Proposition~\ref{prop:orbits} give the homogeneous-space assertion of
Corollary~\ref{thm:C}.
\end{proof}

\begin{remark}
The fibers of the restricted incidence maps are the complements
in \(\Pp^2\) and \(\Pp^1\) of the hyperplane sections imposed by the
two Drinfeld domains.
\end{remark}

\section{An explicit flag and the restricted fibers}\label{sec:example}

Let \(K'/K\) be the unramified extension of degree seven, and choose
\(t\in K'\) whose residue generates the degree-seven residue field
extension. Then \(1,t,\ldots,t^6\) are linearly independent over \(K\).
In the basis of Section~\ref{sec:orbits}, put
\[
u=(1,0,t,t^2),\qquad v=(0,1,t^3,t^5),\qquad
W=K'u+K'v,\qquad L=K'(u+tv).
\]
\begin{proposition}
The flag \(L\subset W\) defines a \(K'\)-point of
\(\Omega_{\Fcal}\). Its images are \(K'\)-rational points of
\(\Omega_{\Tcal}\) and \(\Omega_{\Mcal}\).
\end{proposition}
\begin{proof}
The Pl\"ucker coordinates of \(W\), in the order
\(12,13,14,23,24,34\), are
\[
(1,t^3,t^5,-t,-t^2,t^6-t^5).
\]
They are linearly independent over \(K\): a relation between them is
a polynomial relation in \(t\) of degree at most six; comparing its
coefficients gives zero for each coefficient of the original relation.
Thus \(W\in\Omega_{\Mcal}(K')\). The line has coordinates
\[
[1:t:t+t^4:t^2+t^6].
\]
These four elements are likewise linearly independent over \(K\),
by comparison of coefficients of \(1,t,t^4,t^2,t^6\).
Hence \(L\in\Omega_{\Tcal}(K')\), proving the assertion.
For the chosen symplectic form, \(\omega(u,v)=t^5-t\ne0\), in
agreement with the open-orbit condition.
\end{proof}

If \(F/K\) is finite and
\(\Omega_{\Tcal}(F)\ne\varnothing\), then \([F:K]\ge4\).
If \(\Omega_{\Mcal}(F)\ne\varnothing\), then \([F:K]\ge6\).
These inequalities follow from the required linear independence of four
and six coordinates. In particular, the domains have no \(K\)-rational
points. The flag constructed above is defined over the degree-seven
extension \(K'/K\).

For \(x\in\Omega_{\Mcal}\), write \(F=\mathscr H(x)\) and let
\(W_x\subset V_F\) be the associated plane. The analytic fiber is
\[
q^{-1}(x)=\Pp(W_x)^{\an}\setminus
\bigcup_{\substack{H\subset V\\\dim H=3}}
\bigl(\Pp(W_x)\cap\Pp(H_F)\bigr)^{\an}.
\]
Each removed intersection is an \(F\)-rational point of this projective
line, and its properness was proved in Lemma~\ref{lem:q-surj}.
For \(z\in\Omega_{\Tcal}\), write \(F=\mathscr H(z)\) and
\(L_z=F\ell\). Then
\[
p^{-1}(z)=\Pp(V_F/L_z)^{\an}\setminus
\bigcup_{0\ne\lambda\in(\Lambda^2V)^*}
\{[w]:\lambda(\ell,w)=0\}^{\an}.
\]
Each removed set is an \(F\)-rational projective line, by
Lemma~\ref{lem:p-surj}. These are hyperplanes over \(F\) induced
by hyperplanes of \(V\) and alternating forms on \(V\) defined over
\(K\), respectively. Their coefficients in fiber coordinates can lie
in \(F\).

\section{Comparison with torus-stable domains}\label{sec:stable}

Let \(T\subset\operatorname{SL}(V)\) be a maximal \(K\)-split torus,
that is, a maximal algebraic subgroup isomorphic over \(K\) to a product
of multiplicative groups. A \(K\)-basis \(f_1,\ldots,f_4\) diagonalizes
\(T\). Its character space is
\(\mathbf R^4/\mathbf R(1,1,1,1)\), and the weights on
\(\Lambda^2V\) are
\[
\chi_{ij}=\epsilon_i+\epsilon_j\pmod{\mathbf R(1,1,1,1)},\qquad i<j,
\]
where \(\epsilon_i\) are the coordinate vectors. For a point
\(x\in\Gr(2,V)^{\an}\), evaluate its Pl\"ucker coordinates in
\(\mathscr H(x)\). The numerical criterion for stability with respect
to the Pl\"ucker line bundle says that \(x\) is \(T\)-stable precisely
when zero lies in the interior of the convex hull of those \(\chi_{ij}\)
whose coordinates are nonzero; see \cite[\S1.6]{Voskuil}.
Let \(Y^s\) be the set of points that satisfy this condition for every
maximal \(K\)-split torus of \(\operatorname{SL}(V)\). This is the Berkovich formulation of the stable locus in
Voskuil's construction for this group and line bundle.

\begin{proposition}\label{prop:stable-comparison}
As subsets of \(\Gr(2,V)^{\an}\),
\[
\Omega_{\Mcal}\subsetneq Y^s.
\]
For every nondegenerate alternating form \(\omega\), the intersection
\(Y^s\cap\LG(2,V)^{\an}\) is nonempty.
\end{proposition}
\begin{proof}
Every Pl\"ucker coordinate in any \(K\)-basis is a nonzero
\(K\)-linear functional on \(\Lambda^2V\). At a point of
\(\Omega_{\Mcal}\) all six coordinates are therefore nonzero.
The weights occur in opposite pairs
\(\chi_{12}=-\chi_{34}\), \(\chi_{13}=-\chi_{24}\), and
\(\chi_{14}=-\chi_{23}\); the first member of each pair forms a
basis of the three-dimensional character space. Their convex hull has
zero in its interior. This proves \(\Omega_{\Mcal}\subset Y^s\).

The integral quadric \(D=\LG(2,V)\) spans the hyperplane
\(\Pp(\ker c_\omega)\), by its nondegenerate quadratic equation in
Section~\ref{sec:orbits}. A Pl\"ucker coordinate functional in any
\(K\)-basis is a decomposable alternating form \(f_i^*\wedge f_j^*\)
of rank two. It cannot be proportional to \(c_\omega\), which
corresponds to the rank-four form \(\omega\). Its restriction to
\(D\) is consequently nonzero.
Choose an analytic point \(x\in D^{\an}\) lying over the generic
point of \(D\), using Lemma~\ref{lem:generic-analytic-point}.
All six Pl\"ucker coordinates in every \(K\)-basis are nonzero at
\(x\). The same weight calculation makes \(x\) stable for every
maximal \(K\)-split torus. But \(c_\omega(x)=0\), so
\(x\notin\Omega_{\Mcal}\). This proves both assertions.
\end{proof}

\section{Noncompactness of the quotients}\label{sec:compactness}

Let \(G=\PGL(V)(K)\) act on \(\Pp(\Lambda^2V^*)(K)\) by
\(g\cdot[\beta]=[(g^{-1})^*\beta]\). The stabilizer of \([\omega]\) is \(H_\omega\), the projective
symplectic similitude group.

\begin{lemma}\label{lem:compact-cocompact}
Let \(E\) be a finite-dimensional \(K\)-vector space of dimension at least two, and let
\(\Gamma\subset\PGL(E)(K)\) be torsion-free and discrete.  If
\(\Gamma\backslash\Omega_{\mathrm{Dr}}(E)\) is compact, then
\(\Gamma\backslash\PGL(E)(K)\) is compact.
\end{lemma}

\begin{proof}
Write \(G_E=\PGL(E)(K)\) and \(\mathscr B_E=\mathscr B(\PGL(E),K)\).
The proof of Proposition~\ref{prop:drinfeld_action}(c) gives a
\(G_E\)-equivariant retraction of \(\Omega_{\mathrm{Dr}}(E)\) onto a
closed subspace \(S\), and the canonical homeomorphism
\(S\simeq\mathscr B_E\) is \(G_E\)-equivariant.  The induced map
\(\Gamma\backslash\Omega_{\mathrm{Dr}}(E)\to\Gamma\backslash\mathscr B_E\)
is continuous and surjective, so \(\Gamma\backslash\mathscr B_E\) is
compact.

Since \(\mathscr B_E\) is locally compact, finitely many relatively
compact open subsets of \(\mathscr B_E\) have images covering
\(\Gamma\backslash\mathscr B_E\).  Hence there is a compact
\(F\subset\mathscr B_E\) with \(\Gamma F=\mathscr B_E\).  Fix a vertex
\(x_0\).  Since \(G_E\) acts properly on \(\mathscr B_E\), the set
\(C=\{g\in G_E: gx_0\in F\}\) is compact.  For \(g\in G_E\) write
\(gx_0=\gamma y\) with \(\gamma\in\Gamma\) and \(y\in F\); then
\(\gamma^{-1}g\in C\).  Thus \(G_E=\Gamma C\), and
\(\Gamma\backslash G_E\) is compact.
\end{proof}

\begin{lemma}\label{lem:G-mod-H}
For every nondegenerate alternating form \(\omega\) on \(V\), the space
\(G/H_\omega\) is not compact.  Hence no subgroup of \(H_\omega\) is cocompact in \(G\).
\end{lemma}

\begin{proof}
The continuous orbit map
\(G/H_\omega\to\Pp(\Lambda^2V^*)(K)\) has image
\(G\cdot[\omega]\). If \(G/H_\omega\) were compact, this image
would be compact and therefore closed in the Hausdorff projective space.
In the dual symplectic basis, put
\[
\omega_n=e_1^*\wedge e_4^*+\pi^n e_2^*\wedge e_3^*,\qquad n\ge1.
\]
Each \(\omega_n\) is nondegenerate, hence lies in the
\(\GL(V)(K)\)-orbit of \(\omega\). But \([\omega_n]\) converges to
\([e_1^*\wedge e_4^*]\), which has rank two and lies outside that orbit.
Thus \(G/H_\omega\) is noncompact.
If \(\Gamma\subset H_\omega\), the map
\[
\Gamma\backslash G\longrightarrow G/H_\omega,\qquad
\Gamma g\longmapsto g^{-1}H_\omega
\]
is a well-defined continuous surjection. The noncompactness of its
target excludes compactness of \(\Gamma\backslash G\).
\end{proof}

\begin{lemma}\label{lem:fibers-noncompact}
Let \(z\in\Omega_{\Tcal}\) and \(x\in\Omega_{\Mcal}\), and let \(F\)
be \(\mathscr H(z)\), respectively \(\mathscr H(x)\).  The fibers of
\(p:\Omega_{\Fcal}\to\Omega_{\Tcal}\) over \(z\) and of
\(q:\Omega_{\Fcal}\to\Omega_{\Mcal}\) over \(x\) are nonempty open
subsets of \(\Pp(V_F/L_z)^{\an}\), respectively \(\Pp(W_x)^{\an}\),
different from the whole space.  They are not compact.
\end{lemma}

\begin{proof}
The fibers are nonempty by Lemmas~\ref{lem:p-surj} and~\ref{lem:q-surj},
and open because \(\Omega_{\Fcal}\) is open.  Every \(K\)-rational
hyperplane \(H\subset V\) has scalar extension \(H_F\) meeting
\(W_x\) in a line: the dimension inequality gives a nonzero
intersection, and Lemma~\ref{lem:q-surj} excludes \(W_x\subset H_F\); the point \(\Pp(W_x)\cap\Pp(H_F)\) is not in the
fiber of \(q\).  For \(0\ne\lambda\in(\Lambda^2V)^*\), the line
\(\{[w]:\lambda(\ell,w)=0\}\subset\Pp(V_F/L_z)\) is not in the fiber of
\(p\).  The spaces \(\Pp(V_F/L_z)^{\an}\) and \(\Pp(W_x)^{\an}\) are
connected and Hausdorff.  In a connected Hausdorff space a compact open
subset is also closed, hence empty or the whole space.
\end{proof}

\begin{proof}[Proof of Theorem~\ref{thm:N}]
(a) Proposition~\ref{prop:symmetry} and the local quotient charts
show that \(\mathcal H\) descends for every \(\Gamma\subset H_\omega\).
If \(Y_\Gamma\) were compact, Lemma~\ref{lem:compact-cocompact}
with \(E=V\) would make \(\Gamma\) cocompact in \(G\), contrary to
Lemma~\ref{lem:G-mod-H}.

(b) This is the second assertion of Lemma~\ref{lem:G-mod-H}; by
Proposition~\ref{prop:symmetry}, \(\mathcal H\) is
\(\Gamma\)-invariant exactly when \(\Gamma\subset H_\omega\).
The same stabilizer preserves its restriction to \(\Omega_{\Tcal}\).
Indeed, Lemma~\ref{lem:generic-analytic-point} provides a point of
\(\Omega_{\Tcal}\) over the generic point of \(\Pp(V)\).
Equality of the two restricted distributions implies equality after
extension of the function field to its completed residue field, hence
at the algebraic generic point. The corresponding sections of the
relative Grassmannian of two-planes in \(T_{\Pp(V)}\) then agree
everywhere, since that Grassmannian is separated and \(\Pp(V)\) is
integral.

(c) Let \(\bar z\) be the image of \(z\in\Omega_{\Tcal}\) in
\(\Gamma\backslash\Omega_{\Tcal}\).  This quotient is Hausdorff, so
\(\bar p^{-1}(\bar z)\) is closed.  Since the action is proper, the
orbit \(\Gamma z\) is closed and discrete, and \(p^{-1}(\Gamma z)\) is
the disjoint union of the open and closed subsets \(p^{-1}(\gamma z)\),
\(\gamma\in\Gamma\).  The quotient map
\(\Omega_{\Fcal}\to\Gamma\backslash\Omega_{\Fcal}\) is open, so its
restriction to the saturated set \(p^{-1}(\Gamma z)\) is an open map
onto \(\bar p^{-1}(\bar z)\).  Its restriction to \(p^{-1}(z)\) is
injective, because \(\gamma p^{-1}(z)\cap p^{-1}(z)\ne\varnothing\)
implies \(\gamma z=z\), hence \(\gamma=1\).  Therefore
\(\bar p^{-1}(\bar z)\) is homeomorphic to \(p^{-1}(z)\), which is not
compact by Lemma~\ref{lem:fibers-noncompact}.  The same argument
applies to \(\bar q\). A closed subset of a compact space is
compact, so the flag quotient is noncompact. Proper analytic morphisms
have compact fibers; hence both descended projections are nonproper
for every \(\Gamma\) under consideration.
\end{proof}

Suppose that \(K\) is a finite extension of \(\mathbf Q_p\).
Let \(\Gamma\subset\PGL(V)(K)\) be torsion-free, discrete and
cocompact. Classical uniformization identifies \(\Gamma\backslash\Omega_{\Tcal}\) with
the analytification of a smooth projective variety
\cite{Mustafin,Kurihara}. The descended incidence projections have noncompact
fibers by Theorem~\ref{thm:N}.

\begin{remark}\label{rem:complex-contrast}
Let \(h\) be a quaternionic Hermitian form of signature \((1,1)\)
on \(\mathbf H^2\), viewed as \(\mathbf C^4\), and let
\(\operatorname{Sp}(1,1)\) be its isometry group. This group preserves
a complex symplectic form on \(\mathbf C^4\). On
\(\mathbf{HP}^1=S^4\), it preserves the ball \(B\) consisting of
negative quaternionic lines. The ball is real hyperbolic four-space.
The twistor projection has the quaternionic description
\[
\tau:\mathbf{CP}^3\longrightarrow\mathbf{HP}^1,\qquad
v\mathbf C\longmapsto v\mathbf H;
\]
see \cite[\S2 and \S9]{EellsSalamon}. Each quaternionic linear
\(g\in\operatorname{Sp}(1,1)\) is complex linear and lifts its action
on \(\mathbf{HP}^1\) to
\[
\widetilde g(v\mathbf C)=(gv)\mathbf C,\qquad
\tau\circ\widetilde g=g\circ\tau.
\]
This formula gives a holomorphic group action on \(\mathbf{CP}^3\).
For a torsion-free cocompact discrete subgroup
\(\Gamma\subset\operatorname{Sp}(1,1)\), put
\(\Omega=\tau^{-1}(B)\); compare \cite{SeadeVerjovskyKleinian}.
Stabilizers on \(B\) are compact, so the discrete torsion-free group
acts freely and properly on \(B\), and hence on \(\Omega\).
The quotient is a smooth \(\mathbf{CP}^1\)-bundle over the compact
hyperbolic manifold \(\Gamma\backslash B\). It is therefore compact,
and its holomorphic contact distribution descends from
\(\mathbf{CP}^3\). 
\end{remark}

\begin{proposition}\label{prop:demailly}
Let \(K\) have characteristic zero, and let \(X\) be a smooth
projective threefold over \(K\) with ample canonical bundle.
Neither \(X\) nor \(X^{\an}\) admits a contact structure.
\end{proposition}
\begin{proof}
An algebraic contact distribution and its quotient line bundle, together
with \(X\), descend to a finitely generated subfield \(K_0\subset K\).
Smoothness, the contact condition, and ampleness descend under field
extension. An embedding \(K_0\hookrightarrow\mathbf C\) would thus
give a smooth complex projective contact threefold with ample canonical
bundle, contrary to \cite[Corollary~2]{Demailly}.
An analytic contact subbundle on \(X^{\an}\), its quotient, and the
inclusion morphism algebraize by proper GAGA. Exactness, local freeness,
and nondegeneracy of the bracket can be checked after analytification,
so the algebraized subbundle would be a contact structure on \(X\).
\end{proof}

\section*{Funding}
This work was supported by Proyecto PAPIIT, Direcci\'on General de Asuntos del Personal Acad\'emico, Universidad Nacional Aut\'onoma de M\'exico [grant number IN103324].

\section*{Acknowledgments}
I used ChatGPT (OpenAI) and Claude (Anthropic) for editorial revision, mathematical discussion, and checks of arguments and calculations.  I am responsible for the mathematical statements, proofs, and references.

\bigskip
\noindent
\textsc{Alberto Verjovsky}\\
Instituto de Matem\'aticas, Universidad Nacional Aut\'onoma de M\'exico\\
Unidad Cuernavaca, M\'exico\\
\texttt{albertoverjovsky@gmail.com}

\end{document}